\documentclass[oneside,english]{amsart}
\usepackage[T1]{fontenc}
\usepackage[latin9]{inputenc}
\usepackage{geometry}
\usepackage{float}
\usepackage{mathtools}
\usepackage{amstext}
\usepackage{amsthm}
\usepackage{amssymb}
\usepackage{setspace}

\makeatletter

\providecommand{\tabularnewline}{\\}

\theoremstyle{definition}
\newtheorem*{defn*}{\protect\definitionname}
\theoremstyle{plain}
\newtheorem*{conjecture*}{\protect\conjecturename}
\theoremstyle{definition}
 \newtheorem{example}{\protect\examplename}
\theoremstyle{plain}
\newtheorem{thm}{\protect\theoremname}
\theoremstyle{plain}
\newtheorem{lem}{\protect\lemmaname}

\makeatletter
\usepackage{etoolbox}
\patchcmd{\@maketitle}
  {\ifx\@empty\@dedicatory}
  {\smallskip
    \begin{center}
    \footnotesize
    \begin{tabular}{c}
      Joseph Squillace \\
      University of California, Irvine  \\
      Department of Mathematics  \\
    \end{tabular}
  \end{center}
  \ifx\@empty\@dedicatory}
  {}{}

\usepackage{colortbl}

\makeatother

\usepackage{babel}
\providecommand{\conjecturename}{Conjecture}
\providecommand{\definitionname}{Definition}
\providecommand{\examplename}{Example}
\providecommand{\lemmaname}{Lemma}
\providecommand{\theoremname}{Theorem}

\begin{document}
\title[The Component Counting Process of a Uniform Random Variable]{ On the Dependence of the Component Counting Process of a Uniform
Random Variable}
\thanks{We thank Prof. Michael Cranston and Prof. Nathan Kaplan for their
support in this project.}
\begin{abstract}
We are concerned with the general problem of proving the existence
of joint distributions of two discrete random variables $M$ and $N$
subject to infinitely many constraints of the form $\mathbb{P}\left(M=i,N=j\right)=0$.
In particular, the variable $M$ has a countably infinite range and
the other variable $N$ is uniformly distributed with finite range.
The constraints placed on the joint distributions will require, for
most elements $j$ in the range of $N$, $\mathbb{P}\left(M=i,N=j\right)=0$
for infinitely many values of $i$ in the range of $M$, where the
corresponding values of $i$ depend on $j$. To prove the existence
of such joint distributions, we apply a theorem proved by Strassen
on the existence of joint distributions with prespecified marginal
distributions.

We consider some combinatorial structures that can be decomposed into
components. Given $n\in\mathbb{N}$, consider an assembly, multiset,
or selection $A_{n}$ among elements of $\left[n\right]\coloneqq\left\{ 1,2,\ldots,n\right\} $,
and consider a uniformly distributed random variable $N\left(n\right)$
on $A_{n}$. For each $i\le n$, denote by $C_{i}\left(n\right)$
the number of components of $N\left(n\right)$ of size $i$ so that
$\sum_{i\le n}iC_{i}\left(n\right)=n$. In each of these combinatorial
structures, there exists infinitely many processes $\left(\left(Z_{i}\left(n,x\right)\right)_{i\le n}\right)_{x}$,
indexed by a real parameter $x,$ consisting of non-negative independent
variables $\left(Z_{i}\left(n,x\right)\right)_{i\le n}$ such that
the distribution of the vector $\left(C_{i}\left(n\right)\right)_{i\le n}$
equals the distribution of the vector $\left(Z_{i}\left(n,x\right)\right)_{i\le n}$
conditional on the event $\left\{ \sum_{i\le n}iZ_{i}\left(n,x\right)=n\right\} $.
Let $M\left(n,x\right)$ denote a random variable whose components
are given by $\left(Z_{i}\left(n,x\right)\right)_{i\le n}$. We introduce
the notion of pivot mass which is then combined with Strassen's work
to provide couplings of $M\left(n,x\right)$ and $N\left(n\right)$
with desired properties. For each of these combinatorial structures,
we prove that there exists a real number $x\left(n\right)$ for which
we can couple $M\left(n,x\right)$ and $N\left(n\right)$ with $\sum_{i\le n}\left(C_{i}\left(n\right)-Z_{i}\left(n,x\right)\right)^{+}\le1$
when $x>x\left(n\right)$. We are providing a partial answer to the
question ``how much dependence is there in the process $\left(C_{i}\left(n\right)\right)_{i\le n}?"$
\end{abstract}

\maketitle

\section{Introduction}

Our results regard the component counting process of a discrete uniform
random variable in a combinatorial structure, and these results are
provided by establishing the existence of couplings of random variables.
\begin{defn*}
Let $X$ and $Y$ be random variables defined on probability spaces\footnote{In our discrete setting, we can consider probability spaces $\left(\Omega,\mathcal{F},\mathbb{P}\right)$
of the following form: (i) $\Omega$ is a nonempty set that is finite
or countably infinite; (ii) $\mathcal{F}$ is the power set of $\Omega$;
and (iii) the probability measure $\mathbb{P}$ is defined as $\mathbb{P}\left(E\right)=\sum_{e\in E}p\left(e\right)$
for all $E\in F$, where $p$ is a probability mass function (PMF)
\textendash{} i.e., $p:\Omega\to\left[0,1\right]$ with $\sum_{s\in\Omega}p\left(s\right)=1$. } $\left(\Omega_{X},\mathcal{F}_{X},\mathbb{P}_{X}\right)$ and $\left(\Omega_{Y},\mathcal{F}_{Y},\mathbb{P}_{Y}\right)$.\footnote{The probability measure $\mathbb{P}_{X}$ is defined by $\mathbb{P}_{X}\left(i\right)=\mathbb{P}\left(X=i\right)$.}
A \textbf{coupling }of $X$ and $Y$ is a probability space $\left(\Omega,\mathcal{F},\mathbb{P}\right)$
in which there exists random variables $X'$ and $Y'$ such that $X'$
has the same distribution as $X$ and $Y'$ has the same distribution
as $Y$.
\end{defn*}
\noindent For each of the random variables $X$ considered in this
paper, $X$ and $X'$ will share the same range. Thus, for each coupling
of $X$ and $Y,$ the definition implies that there exists a joint
probability mass function $p\left(x,y\right)\coloneqq\mathbb{P}\left(X'=x,Y'=y\right)$
whose marginal distributions satisfy 
\begin{eqnarray*}
\sum_{y:p\left(x,y\right)>0}p\left(x,y\right) & = & \mathbb{P}_{X}\left(x\right),\\
\sum_{x:p\left(x,y\right)>0}p\left(x,y\right) & = & \mathbb{P}_{Y}\left(y\right).
\end{eqnarray*}
Equivalently\footnote{When describing a particular coupling of $X$ and $Y$, we often write
$X$ and $Y$ instead of $X'$ and $Y',$ respectively.}, $\mathbb{P}_{X'}\left(x\right)=\mathbb{P}_{X}\left(x\right)$ and
$\mathbb{P}_{Y'}\left(y\right)=\mathbb{P}_{Y}\left(y\right)$ for
all $x$ in the range of $X$ and all $y$ in the range of $Y$. In
particular, we provide couplings, with some constraints, of a uniform
random variable $N$, necessarily consisting of a dependent component
process, with another random variable $M$ having the following properties:
\begin{enumerate}
\item $M$ has infinite range.
\item $M$ and $N$ have the same number of components (but the sizes of
their corresponding components need not be equal).
\item The components of $M$ are independent and nonnegative.
\end{enumerate}
\noindent The constraints imposed on our couplings are motivated by
the following conjecture, proposed by Richard Arratia in $\S$2.2
of \cite{key-1}, which we now describe. Consider a uniformly distributed
variable $N\left(n\right)\in\left[n\right]$ with prime factorization
\begin{eqnarray*}
N\left(n\right) & = & \prod_{p\le n}p^{C_{p}\left(n\right)}.
\end{eqnarray*}
It can be shown that the prime power process $\left(C_{p}\left(n\right)\right)_{p\le n}$
converges in distribution to a process $\left(Z_{p}\right)_{p\le n}$
of independent variables where $Z_{p}$ is a geometric random variable
of parameter $\frac{1}{p}$ and range $\mathbb{Z}_{\ge0}$, for each
prime $p\le n.$ Defining 
\begin{eqnarray*}
M\left(n\right) & = & \prod_{p\le n}p^{Z_{p}},
\end{eqnarray*}
we state Arratia's conjecture.
\begin{conjecture*}
For all $n\ge1$, it is possible to construct $N\left(n\right)$ uniformly
distributed from $1$ to $n$, $M\left(n\right)$ and a prime $P\left(n\right)$
such that 
\begin{eqnarray*}
\text{always } &  & N\left(n\right)\text{ divides }M\left(n\right)P\left(n\right).
\end{eqnarray*}
\end{conjecture*}
\noindent Equivalently, the conjecture states that there exists a
coupling of $M\left(n\right)$ and $N\left(n\right)$ such that we
always have 
\begin{eqnarray*}
\sum_{p\le n}\left(C_{p}\left(n\right)-Z_{p}\right)^{+} & \le & 1.
\end{eqnarray*}
This is also equivalent to the existence of a joint probability mass
function $p\left(\cdot,\cdot\right)$ with marginals corresponding
to $M\left(n\right)$ and $N\left(n\right)$ such that $p\left(\cdot,\cdot\right)=0$
when 
\begin{eqnarray*}
\sum_{p\le n}\left(C_{p}\left(n\right)-Z_{p}\right)^{+} & > & 1.
\end{eqnarray*}
We impose an analogous constraint on the couplings provided in this
paper, but now we will point out some differences between these couplings
and the coupling conjectured by Arratia. First, we drop the requirement
that $N\left(n\right)\in\left[n\right];$ rather, from now on we let
$N\left(n\right)$ denote a uniform variable in a combinatorial structure
over $\left[n\right]$ (these structures are defined in $\S$1.1).
Instead of a prime power process $\left(C_{p}\left(n\right)\right)_{p\le n}$,
we consider a component counting process $\left(C_{i}\left(n\right)\right)_{1\le i\le n}$
(here $i$ is any positive integer less than or equal $n$) of $N\left(n\right)$
which satisfies $\sum_{i\le n}iC_{i}\left(n\right)=n$ \textendash{}
the latter equation is not always true for the prime power process
$\left(C_{p}\left(n\right)\right)_{p\le n}$ of a uniformly distributed
variable over $\left[n\right]$. In Arratia's conjecture, there is
a natural candidate for $M\left(n\right)$ since the prime power process
$\left(C_{p}\left(n\right)\right)_{p\le n}$ converges in distribution
to the process $\left(Z_{p}\right)_{p\le n}$ described above. However,
in each of the examples considered in this paper, we take advantage
of the fact that in either an assembly, multiset, or selection, there
exists infinitely many processes $\left(Z_{i}\left(n,x\right)\right)_{i\le n}$,
indexed by some positive real parameter $x$, consisting of independent
variables $Z_{i}\left(n,x\right),$ which furnish natural candidates
for a random variable $M\left(n,x\right)$ to be compared with a uniform
random variable $N\left(n\right)$ (see equation (\ref{eq:1}) below). 

The combinatorial structures listed in $\S$1.1 provide the frameworks
in which we obtain our couplings. Theorem 1, the main result of this
paper, is stated in $\S$1.2. In $\S$2, we describe how our constraints
force a significant proportion of the entries of a prospective joint
mass distribution of our variables to be $0$. In $\S$3, we introduce
the notion of pivot mass, which depends on the constraints placed
on the desired joint distribution. Some properties of the pivot mass
are proved in $\S$3 and $\S$4. In $\S$5, we apply results on the
pivot mass and a theorem proved by Strassen to prove Theorem 1, thereby
proving the existence of our couplings.

\subsection{Three Major Combinatorial Structures}

All couplings constructed in this paper involve a discrete uniform
random variable in any one of the following three combinatorial classes.
An \textbf{assembly} $A_{n}$ is an example of a combinatorial structure
in which the set $\left[n\right]$ is partitioned into blocks and
for each block of size $i$ one of $m_{i}$ possible structures is
chosen. An example of an assembly is the collection of set partitions
of $\left[n\right]$, in which case $m_{i}=1$ for $i\le n$ (since
the order of the elements in a particular block is irrelevant \textendash{}
i.e., once $i$ numbers $n_{1},\ldots,n_{i}\in\left[n\right]$ are
chosen and placed in a box of size $i$, there is a unique block consisting
of these $i$ elements). Moreover, for set partitions of $\left[n\right],$
we have $\#A_{n}=B_{n},$ the $n$th Bell number. Another example
of an assembly is the set $S_{n}$ of permutations of $\left[n\right]$,
in which case $m_{i}=\left(i-1\right)!$ (since there are $\left(i-1\right)!$
distinct cycles of length $i$ among $i$ chosen numbers $n_{1},\ldots,n_{i}\in\left[n\right]$)
for $i\le n$. Further, for permutations of $\left[n\right],$ we
have $\#A_{n}=n!$. A \textbf{multiset }$A_{n}$ is a pair $\left(\left[n\right],m\right),$
where $m:A\to\mathbb{N}$ is a function that gives the multiplicity
$m\left(a\right)$ of each element $a\in\left[n\right]$. Equivalently
(see Meta-example 2.2 of $\S$2.2 of \cite{key-2}), the integer $n$
is partitioned into parts, and for each part of size $i$, one of
the $m_{i}$ objects of weight $i$ is chosen. In the example of integer
partitions of a positive integer $n$, we have $m_{i}=1$ (for each
part of size $i$, we have only $m_{i}=1$ choice for the size of
$i$) for $1\le i\le n$. When $A_{n}$ is the set of integer partitions
of $n$, we have $\#A_{n}=p\left(n\right),$ where $p$ is the integer
partition function.\textbf{ Selections} are similar to multisets,
but now we require all parts to be distinct. An example of a selection
is the set of all integer partitions of a positive integer $n$ with
distinct parts. In the case of integer partitions with distinct parts,
we have $\#A_{n}=q\left(n\right),$ where $q$ is the integer partition
function with distinct parts. To simplify the notation, let us define
$k_{n}\coloneqq\#A_{n}$ for each of these structures.

These three structures are characterized by the following generating
relations between $k_{n}$ and $m_{i}$. Assemblies are characterized
by 
\begin{eqnarray*}
\sum_{n\ge0}\frac{\left(k_{n}\right)z^{n}}{n!} & = & \exp\left(\sum_{i\ge1}\frac{m_{i}z^{i}}{i!}\right),
\end{eqnarray*}
multisets are characterized by 
\begin{eqnarray*}
\sum_{n\ge0}\left(k_{n}\right)z^{n} & = & \prod_{i\ge1}\left(1-z^{i}\right)^{-m_{i}},
\end{eqnarray*}
and selections are characterized by 
\begin{eqnarray*}
\sum_{n\ge0}\left(k_{n}\right)z^{n} & = & \prod_{i\ge1}\left(1+z^{i}\right)^{m_{i}}
\end{eqnarray*}
($\S$2.2 of \cite{key-2}). Revisiting the example of an assembly
in which $A_{n}$ denotes the set of all set partitions of $\left[n\right]$
(so that $m_{i}=1$ for $1\le i\le n$), it is known (e.g., pp. 20-23
of \cite{key-5}) that the $n$th Bell number $B_{n}$ satisfies the
generating equation $\sum_{n\ge0}\frac{B_{n}}{n!}z^{n}=\text{exp}\left(e^{z}-1\right)$,
and the right hand side may be expressed as $\exp\left(\sum_{i\ge1}\frac{z^{i}}{i!}\right)$.

\subsection{Couplings of Random Variables}

In each of the assembly, multiset, and selection settings, our methods
of arriving at our desired couplings are similar. We start by considering
$N\left(n\right)\sim\text{Unif}\left(A_{n}\right)$. Given $i\le n$,
if we denote by $C_{i}\left(n\right)$ the number components of $N\left(n\right)$
of size $i$, then $0\le C_{i}\left(n\right)\le n$ and $\sum_{i\le n}iC_{i}\left(n\right)=n$.
In particular, the variables $C_{i}\left(n\right),i\le n$, are dependent
and their distributions are determined by the uniform variable $N\left(n\right)$.
The process $\left(C_{i}\left(n\right)\right)_{i\le n}=\left(C_{1}\left(n\right),\ldots,C_{n}\left(n\right)\right)$
is called the \textbf{component counting process }of $N\left(n\right)$.
\begin{example}
In the example $A_{n}=S_{n},$ the term $C_{i}\left(n\right)$ is
the number of cycles of $N\left(n\right)$ of length $i$, and $\left(C_{1}\left(n\right),\ldots,C_{n}\left(n\right)\right)$
is often referred to as the \textit{cycle type} of $N\left(n\right)$.
In the example for which $A_{n}$ is the collection of set partitions
of $\left[n\right],$ $C_{i}\left(n\right)$ is the number of blocks
of $N\left(n\right)$ of size $i$. In the example for which $A_{n}$
is the set of integer partitions of $n$, $C_{i}\left(n\right)$ is
the number of $i'$s in the integer partition $N\left(n\right)$ of
$n$.
\end{example}
In each of these combinatorial settings, there exists an infinite
family $\left(\left(Z_{i}\left(n,x\right)\right)_{i\le n}\right)_{x}$,
parametrized by positive values of $x$ (specifically, $x>0$ for
assemblies, $x\in\left(0,1\right)$ for multisets, and $x\in\left(0,\infty\right)$
for selections) of infinite sequences $\left(Z_{i}\left(n,x\right)\right)_{i\le n}$
of nonnegative integer-valued independent random variables $Z_{i}\left(n,x\right)$
for which 
\begin{equation}
\mathcal{L}\left(C_{1}\left(n\right),\ldots,C_{n}\left(n\right)\right)=\mathcal{L}\left(Z_{1}\left(n,x\right),\ldots,Z_{n}\left(n,x\right)\Bigg\vert\sum_{i\le n}iZ_{i}\left(n,x\right)=n\right)\label{eq:1}
\end{equation}

\noindent ($\S$2.3 of \cite{key-2}). Equation (\ref{eq:1}) states
that the probability that the vector $\left(C_{1}\left(n\right),\ldots,C_{n}\left(n\right)\right)$
belongs to some region $\Gamma\in\mathbb{R}^{n}$ (where $\Gamma$
is an element of the $n$-fold direct product $\prod_{i\le n}\mathcal{B}\left(\mathbb{R}\right)$
of the Borel $\sigma-$algebra on $\mathbb{R}$) is the same as the
conditional probability that $\left(Z_{1}\left(n,x\right),\ldots,Z_{n}\left(n,x\right)\right)$
belongs to $\Gamma$ if we condition on the event $\left\{ \sum_{i\le n}iZ_{i}\left(n,x\right)=n\right\} $.
For a fixed $x,$ we consider another random variable $M\left(n,x\right)$
whose component counting process\footnote{For fixed $x$, since the variables $Z_{i}\left(n,x\right),i\le n,$
are independent, it is not always true that $\sum_{i\le n}iZ_{i}\left(n,x\right)=n$.
Therefore, the variable $M\left(n,x\right)$ does not always correspond
to an element of $A_{n}$.} is given by $\left(Z_{i}\left(n,x\right)\right)_{i\le n}$, so the
distribution of $M\left(n,x\right)$ is determined by the independent
process $\left(Z_{i}\left(n,x\right)\right)_{i\le n}$. The main result
of this paper is the following theorem.\footnote{To simplify the notation, we will sometimes (Figure 2, Theorem 2,
and $\S\S$ 4-5) replace $Z_{i}\left(n,x\right)$ with $Z_{i}$, replace
$C_{i}\left(n\right)$ with $C_{i}$, replace $N\left(n\right)$ with
$N$, and replace $M\left(n,x\right)$ with $M$.}
\begin{thm}
Let $n\in\mathbb{N}$ and suppose $A_{n}$ denotes an assembly, multiset,
or a selection among elements of $\left[n\right]$. Given $N\left(n\right)\sim\text{Unif}\left(A_{n}\right)$
with component counting process $\left(C_{i}\left(n\right)\right)_{i\le n}$,
there exists a positive real number $x\left(n\right)$ for which,
when $x>x\left(n\right)$, there exists a process $\left(Z_{i}\left(n,x\right)\right)_{i\le n}$
of non-negative independent random variables satisfying $\left(1\right)$
such that we can couple $M\left(n,x\right)$ and $N\left(n\right)$
with 
\begin{equation}
\sum_{i\le n}\left(C_{i}\left(n\right)-Z_{i}\left(n,x\right)\right)^{+}\le1.\label{eq:2}
\end{equation}
\end{thm}

\section{The Joint Mass Distribution of $\left(M\left(n,x\right),N\left(n\right)\right)$}

For some fixed value of $x$, if we are to successively construct
a joint probability mass function $p\left(\cdot,\cdot\right)$ with
marginal distributions corresponding to $M\left(n,x\right)$ and $N\left(n\right)$
for which inequality (\ref{eq:2}) holds, we must ensure that $\mathbb{P}\left(M\left(n,x\right)=\cdot,N\left(n\right)=\cdot\right)=0$
when 
\begin{eqnarray*}
 & \sum_{i\le n}\left(C_{i}\left(n\right)-Z_{i}\left(n,x\right)\right)^{+}>1.
\end{eqnarray*}
We can index the joint distribution by using the range of $N\left(n\right)$
and the range of $M\left(n,x\right)$ for the column labels and row
labels, respectively. In particular, we can label the columns with
the range of $\left(C_{i}\left(n\right)\right)_{i\le n}$ in lexicographic
order. Since we have infinitely many row labels, for each $m\in\mathbb{Z}_{\ge0},$
we apply the lexicographic ordering on all elements $\left(m_{1},\ldots,m_{n}\right)\in\left(\mathbb{Z}_{\ge0}\right)^{n}$
with $\sum_{i\le n}m_{i}=m$, starting with $m=0$ (we start with
$m=0$ since the $Z_{i}\left(n,x\right)$'s are non-negative). With
respect to this ordering, we will often enumerate the columns by $1,2,\ldots,k_{n}$
and the rows by $1,2,\ldots$. 

The following example shows that it is possible for several elements
of $A_{n}$ to have the same component process (hence the same column
label). Note that in the setting of Arratia's conjecture, it is impossible
for two columns to have the same label \textendash{} the uniqueness
of prime factorization in $\mathbb{N}$ ensures that each $\left(C_{p}\left(n\right)\right)_{p\le n}$
uniquely determines $N\left(n\right)$.
\begin{example}
Fix $n=3$ and consider the assembly $A_{3}=S_{3}$ of permutations
of $\left\{ 1,2,3\right\} $. The elements of $S_{3}$ are $1,\left(1\,2\right),\left(1\,3\right),\left(2\,3\right),\left(1\,2\,3\right),\left(1\,3\,2\right)$,
and their respective component counts are $\left(3,0,0\right),\left(1,1,0\right),\left(1,1,0\right),$

\noindent $\left(1,1,0\right),\left(0,0,1\right),\left(0,0,1\right)$.
For any $n\in\mathbb{N},$ Cauchy proved that there are $\frac{n!}{\left(i_{1}!\cdots i_{n!}1^{i_{1}}\cdots n^{i_{n}}\right)}$
permutations in $S_{n}$ with cycle type $\left(i_{1},\ldots,i_{n}\right)$,
so this gives the number of elements in $S_{n}$ with the component
counting process $\left(C_{i}\left(n\right)\right)_{i\le n}=\left(i_{1},\ldots,i_{n}\right)$.
\end{example}
For our purposes, when we have multiple columns with the same component
counting process, we enumerate these columns in any order. The reason
that we do not combine these into one column with larger probability
mass is due to the fact we are coupling $M\left(n,x\right)$ and $N\left(n\right)$
instead of coupling the two processes $\left(C_{i}\left(n\right)\right)_{i\le n}$
and $\left(Z_{i}\left(n,x\right)\right)_{i\le n}$ \textendash{} i.e.,
two columns with the same label correspond to different values of
$N\left(n\right)$. For the interested reader, equations $\left(2.2\right),\left(2.3\right),$
and $\left(2.4\right)$ in $\S$2.2 of \cite{key-2} give the number
of columns with a given column label $\left(a_{1},\ldots,a_{n}\right)$
for each of our combinatorial structures.

In each of these three settings, there are additional constraints
on any joint probability mass function of $M\left(n,x\right)$ and
$N\left(n\right)$ since the marginal distributions are known:
\begin{itemize}
\item The sum along column $N\left(n\right)=j,1\le j\le k_{n}$, is 
\begin{eqnarray*}
\sum_{a_{i}\ge0,1\le i\le n}\mathbb{P}\left(N\left(n\right)=j,\left(Z_{i}\left(n,x\right)\right)_{i\le n}=\left(a_{i}\right)_{i\le n}\right) & = & \mathbb{P}\left(N\left(n\right)=j\right)\\
 & = & 1/k_{n}.
\end{eqnarray*}
\item The sum along the row $M\left(n,x\right)=m$, $m\in\mathbb{N},$ labeled
$\left(Z_{i}\left(n,x\right)\right)_{i\le n}=\left(m_{i}\right)_{i\le n}$
is 
\begin{eqnarray*}
\sum\limits _{\substack{a_{i}\ge0,0\le i\le n,\\
\sum_{i\le n}ia_{i}=n
}
}\mathbb{P}\left(\left(C_{i}\left(n\right)\right)_{i\le n}=\left(a_{i}\right)_{i\le n},\left(Z_{i}\left(n,x\right)\right)_{i\le n}=\left(m_{i}\right)_{i\le n}\right) & = & \mathbb{P}\left(\left(Z_{i}\left(n,x\right)\right)_{i\le n}=\left(m_{i}\right)_{i\le n}\right)\\
 & = & \prod_{i\le n}\mathbb{P}\left(Z_{i}\left(n,x\right)=m_{i}\right),
\end{eqnarray*}
\end{itemize}
where the latest equation is due to the independence of the process
$\left(Z_{i}\left(n,x\right)\right)_{i\le n}$.

\section{Pivot Mass}

Given columns $j$ and $k$, with corresponding components $\left(C_{i}\left(n\right)\right)_{i\le n}$
and $\left(C_{i}'\left(n\right)\right)_{i\le n}$, we seek a way to
compare the corresponding sets of row labels in which column $j$
or $k$ must be $0$. Any of our desired couplings has the property
that column $\left(C_{i}\left(n\right)\right)_{i\le n}$ has a zero
in row $\left(Z_{i}\left(n,x\right)\right)_{i\le n}$ when (\ref{eq:2})
is violated, so we compare the probability measures of the sets $\left\{ \left(Z_{i}\left(n,x\right)\right)_{i\le n}:\sum_{i\le n}\left(C_{i}\left(n\right)-Z_{i}\left(n,x\right)\right)^{+}>1\right\} $
and $\left\{ \left(Z_{i}\left(n,x\right)\right)_{i\le n}:\sum_{i\le n}\left(C_{i}'\left(n\right)-Z_{i}\left(n,x\right)\right)^{+}>1\right\} $.
I.e., (\ref{eq:2}) is true for all of our desired couplings, so we
measure the probability that $M\left(n,x\right)$ takes on a value
$i$ for which column $j$ or $k$ has a required $0$ in row $i$.
This motivates the following definition.
\begin{defn*}
We call the pair $\left(i,j\right)$, corresponding to the $i$th
row label $\left(Z_{i}\left(x\right)\right)_{i\le n}$ and the $j$th
column label $\left(C_{i}\left(n\right)\right)_{i\le n}$, a \textbf{pivot}\textit{
}if $\sum_{i\le n}\left(C_{i}\left(n\right)-Z_{i}\left(x\right)\right)^{+}>1$.
Denote the set of all pivots by $P$. The \textbf{pivot mass} in column
$N\left(n\right)=j$ is defined as 
\begin{eqnarray*}
\mathcal{PM}\left(j\right) & =\mathcal{PM}_{\left(n,x\right)}\left(j\right) & \coloneqq\sum\limits _{\substack{i:\left(i,j\right)\in P}
}\mathbb{P}\left(M\left(n,x\right)=i\right).
\end{eqnarray*}
Given a subset $L\left(n\right)$ of column labels of $\left[n\right]$,
the pivot mass in $L\left(n\right)$ is defined as 
\begin{eqnarray*}
\mathcal{PM}\left(L\left(n\right)\right) & =\mathcal{PM}_{\left(n,x\right)}\left(L\left(n\right)\right) & \coloneqq\sum\limits _{\substack{i:\left(i,j\right)\in P\text{ }\\
\forall j\in L\left(n\right)
}
}\mathbb{P}\left(M\left(n,x\right)=i\right).
\end{eqnarray*}
\end{defn*}
\noindent If we define $I_{j}\coloneqq\left\{ \text{row labels }i:\left(i,j\right)\in P\right\} $,
$I_{L\left(n\right)}\coloneqq\left\{ \text{row labels }i:\left(i,j\right)\in P\text{ for all }j\in L\left(n\right)\right\} $
and let $\mathbb{P}_{M}$ denote the PMF of $M\left(n,x\right),$
then 
\[
\mathcal{PM}\left(j\right)=\mathbb{P}_{M}\left(I_{j}\right)
\]
and 
\[
\mathcal{PM}\left(L\left(n\right)\right)=\mathbb{P}_{M}\left(I_{L\left(n\right)}\right).
\]

\noindent Theorem 2 gives a formula for $\mathcal{PM}\left(j\right).$
Fortunately, due to the role of the parameter $x$, it is not necessary
to derive a formula for $\mathcal{PM}\left(L\left(n\right)\right)$
in order to prove Theorem 1. The fact that $\mathcal{PM}\left(L\left(n\right)\right)\le\mathcal{PM}\left(j\right)$
for any $j\in L\left(n\right)$ will be sufficient.

\begin{figure}[H]
\caption{If $\left(i_{0},j_{0}\right)$ is a pivot, then our desired joint
distribution table should have a $0$ in the $\left(i_{0},j_{0}\right)$
entry.}

\phantom{+}

$\phantom{-}$
\centering{}{\small{}}%
\begin{tabular}{|c|l|l|c|c|c|c|c|c|c|c|c|}
\cline{3-10} \cline{4-10} \cline{5-10} \cline{6-10} \cline{7-10} \cline{8-10} \cline{9-10} \cline{10-10} \cline{12-12} 
\multicolumn{1}{c}{} & {\small{}$N\left(n\right)$} & {\small{}$1$} & {\small{}$2$} & {\small{}$\cdots$} & {\small{}$j$} & {\small{}$\cdots$} & {\small{}$j_{0}$} & {\small{}$\cdots$} & {\small{}$k_{n}$} &  & {\small{}Row sum}\tabularnewline
\cline{3-10} \cline{4-10} \cline{5-10} \cline{6-10} \cline{7-10} \cline{8-10} \cline{9-10} \cline{10-10} \cline{12-12} 
\multicolumn{1}{c}{{\small{}$M\left(n,x\right)$}} & \multicolumn{1}{l}{} & \multicolumn{1}{l}{} & \multicolumn{1}{c}{} & \multicolumn{1}{c}{} & \multicolumn{1}{c}{} & \multicolumn{1}{c}{} & \multicolumn{1}{c}{} & \multicolumn{1}{c}{} & \multicolumn{1}{c}{} & \multicolumn{1}{c}{} & \multicolumn{1}{c}{}\tabularnewline
\cline{1-1} \cline{3-10} \cline{4-10} \cline{5-10} \cline{6-10} \cline{7-10} \cline{8-10} \cline{9-10} \cline{10-10} \cline{12-12} 
{\small{}$1$} &  &  &  &  &  &  &  &  &  &  & {\small{}$\mathbb{P}\left(M\left(n,x\right)=1\right)$}\tabularnewline
\cline{1-1} \cline{3-10} \cline{4-10} \cline{5-10} \cline{6-10} \cline{7-10} \cline{8-10} \cline{9-10} \cline{10-10} \cline{12-12} 
{\small{}$2$} &  &  &  &  &  &  &  &  &  &  & {\small{}$\mathbb{P}\left(M\left(n,x\right)=2\right)$}\tabularnewline
\cline{1-1} \cline{3-10} \cline{4-10} \cline{5-10} \cline{6-10} \cline{7-10} \cline{8-10} \cline{9-10} \cline{10-10} \cline{12-12} 
{\small{}$\vdots$} &  &  &  &  &  &  &  &  &  &  & {\small{}$\vdots$}\tabularnewline
\cline{1-1} \cline{3-10} \cline{4-10} \cline{5-10} \cline{6-10} \cline{7-10} \cline{8-10} \cline{9-10} \cline{10-10} \cline{12-12} 
{\small{}$i_{0}$} &  &  &  &  &  &  & {\small{}$0$} &  &  &  & {\small{}$\mathbb{P}\left(M\left(n,x\right)=i_{0}\right)$}\tabularnewline
\cline{1-1} \cline{3-10} \cline{4-10} \cline{5-10} \cline{6-10} \cline{7-10} \cline{8-10} \cline{9-10} \cline{10-10} \cline{12-12} 
{\small{}$\vdots$} &  &  &  &  &  &  &  &  &  &  & {\small{}$\vdots$}\tabularnewline
\cline{1-1} \cline{3-10} \cline{4-10} \cline{5-10} \cline{6-10} \cline{7-10} \cline{8-10} \cline{9-10} \cline{10-10} \cline{12-12} 
{\small{}$i$} &  &  &  &  & {\small{}$\mathbb{P}\left(M\left(n,x\right)=i,N\left(n\right)=j\right)$} &  &  &  &  &  & {\small{}$\mathbb{P}\left(M\left(n,x\right)=i\right)$}\tabularnewline
\cline{1-1} \cline{3-10} \cline{4-10} \cline{5-10} \cline{6-10} \cline{7-10} \cline{8-10} \cline{9-10} \cline{10-10} \cline{12-12} 
{\small{}$\vdots$} &  &  &  &  &  &  &  &  &  &  & {\small{}$\vdots$}\tabularnewline
\cline{1-1} \cline{3-10} \cline{4-10} \cline{5-10} \cline{6-10} \cline{7-10} \cline{8-10} \cline{9-10} \cline{10-10} \cline{12-12} 
\multicolumn{1}{c}{} & \multicolumn{1}{l}{} & \multicolumn{1}{l}{} & \multicolumn{1}{c}{} & \multicolumn{1}{c}{} & \multicolumn{1}{c}{} & \multicolumn{1}{c}{} & \multicolumn{1}{c}{} & \multicolumn{1}{c}{} & \multicolumn{1}{c}{} & \multicolumn{1}{c}{} & \multicolumn{1}{c}{}\tabularnewline
\cline{1-1} \cline{3-10} \cline{4-10} \cline{5-10} \cline{6-10} \cline{7-10} \cline{8-10} \cline{9-10} \cline{10-10} 
{\small{}Column sum} & {\small{}$\phantom{\bigg(}$} & {\tiny{}$\frac{1}{k_{n}}$} & {\tiny{}$\frac{1}{k_{n}}$} & {\tiny{}$\cdots$} & {\tiny{}$\frac{1}{k_{n}}$} & {\tiny{}$\cdots$} & {\tiny{}$\frac{1}{k_{n}}$} & {\tiny{}$\cdots$} & {\tiny{}$\frac{1}{k_{n}}$} & \multicolumn{1}{c}{} & \multicolumn{1}{c}{}\tabularnewline
\cline{1-1} \cline{3-10} \cline{4-10} \cline{5-10} \cline{6-10} \cline{7-10} \cline{8-10} \cline{9-10} \cline{10-10} 
\end{tabular}{\small\par}
\end{figure}

\begin{example}
\begin{singlespace}
\noindent Revisiting the example $A_{3}=S_{3}$, let us illustrate
some key features of a desired joint mass distribution of $\left(M\left(3,x\right),N\left(3\right)\right)$.
\begin{figure}[H]
\caption{A desired coupling of $M\left(3,x\right)$ and $N\left(3\right)$
should have a zero at any location $\left(\left(Z_{i}\left(3,x\right)\right)_{i\le3},\left(C_{i}\left(3\right)\right)_{i\le3}\right)$
satisfying $\sum_{i\le3}\left(C_{i}\left(3\right)-Z_{i}\left(3,x\right)\right)^{+}>1$.}
$\phantom{-}$
\centering{}{\tiny{}}%
\begin{tabular}{|c|l|c|c|c|c|c|c|c|c|}
\cline{3-8} \cline{4-8} \cline{5-8} \cline{6-8} \cline{7-8} \cline{8-8} \cline{10-10} 
\multicolumn{1}{c}{} & {\tiny{}$N$} & {\tiny{}$\left(0,0,1\right)$} & {\tiny{}$\left(0,0,1\right)$} & {\tiny{}$\left(1,1,0\right)$} & {\tiny{}$\left(1,1,0\right)$} & {\tiny{}$\left(1,1,0\right)$} & {\tiny{}$\left(3,0,0\right)$} &  & {\tiny{}Row sum}\tabularnewline
\cline{3-8} \cline{4-8} \cline{5-8} \cline{6-8} \cline{7-8} \cline{8-8} \cline{10-10} 
\multicolumn{1}{c}{{\tiny{}$M$}} & \multicolumn{1}{l}{} & \multicolumn{1}{c}{} & \multicolumn{1}{c}{} & \multicolumn{1}{c}{} & \multicolumn{1}{c}{} & \multicolumn{1}{c}{} & \multicolumn{1}{c}{} & \multicolumn{1}{c}{} & \multicolumn{1}{c}{}\tabularnewline
\cline{1-1} \cline{3-8} \cline{4-8} \cline{5-8} \cline{6-8} \cline{7-8} \cline{8-8} \cline{10-10} 
{\tiny{}$\left(0,0,0\right)$} &  &  &  & {\tiny{}$0$} & {\tiny{}$0$} & {\tiny{}$0$} & {\tiny{}$0$} &  & {\tiny{}$\mathbb{P}\left(\left(Z_{i}\right)_{i\le3}=\left(0,0,0\right)\right)$}\tabularnewline
\cline{1-1} \cline{3-8} \cline{4-8} \cline{5-8} \cline{6-8} \cline{7-8} \cline{8-8} \cline{10-10} 
{\tiny{}$\left(0,0,1\right)$} &  &  &  & {\tiny{}$0$} & {\tiny{}$0$} & {\tiny{}$0$} & {\tiny{}$0$} &  & {\tiny{}$\mathbb{P}\left(\left(Z_{i}\right)_{i\le3}=\left(0,0,1\right)\right)$}\tabularnewline
\cline{1-1} \cline{3-8} \cline{4-8} \cline{5-8} \cline{6-8} \cline{7-8} \cline{8-8} \cline{10-10} 
{\tiny{}$\left(0,1,0\right)$} &  &  &  &  &  &  & {\tiny{}$0$} &  & {\tiny{}$\mathbb{P}\left(\left(Z_{i}\right)_{i\le3}=\left(0,1,0\right)\right)$}\tabularnewline
\cline{1-1} \cline{3-8} \cline{4-8} \cline{5-8} \cline{6-8} \cline{7-8} \cline{8-8} \cline{10-10} 
{\tiny{}$\left(1,0,0\right)$} &  &  &  &  &  &  & {\tiny{}$0$} &  & {\tiny{}$\mathbb{P}\left(\left(Z_{i}\right)_{i\le3}=\left(1,0,0\right)\right)$}\tabularnewline
\cline{1-1} \cline{3-8} \cline{4-8} \cline{5-8} \cline{6-8} \cline{7-8} \cline{8-8} \cline{10-10} 
{\tiny{}$\left(0,0,2\right)$} &  &  &  & {\tiny{}$0$} & {\tiny{}$0$} & {\tiny{}$0$} & {\tiny{}$0$} &  & {\tiny{}$\mathbb{P}\left(\left(Z_{i}\right)_{i\le3}=\left(0,0,2\right)\right)$}\tabularnewline
\cline{1-1} \cline{3-8} \cline{4-8} \cline{5-8} \cline{6-8} \cline{7-8} \cline{8-8} \cline{10-10} 
{\tiny{}$\left(0,1,1\right)$} &  &  &  &  &  &  & {\tiny{}$0$} &  & {\tiny{}$\mathbb{P}\left(\left(Z_{i}\right)_{i\le3}=\left(0,1,1\right)\right)$}\tabularnewline
\cline{1-1} \cline{3-8} \cline{4-8} \cline{5-8} \cline{6-8} \cline{7-8} \cline{8-8} \cline{10-10} 
{\tiny{}$\left(0,2,0\right)$} &  &  &  &  &  &  & {\tiny{}$0$} &  & {\tiny{}$\mathbb{P}\left(\left(Z_{i}\right)_{i\le3}=\left(0,2,0\right)\right)$}\tabularnewline
\cline{1-1} \cline{3-8} \cline{4-8} \cline{5-8} \cline{6-8} \cline{7-8} \cline{8-8} \cline{10-10} 
{\tiny{}$\left(1,0,1\right)$} &  &  &  &  &  &  & {\tiny{}$0$} &  & {\tiny{}$\mathbb{P}\left(\left(Z_{i}\right)_{i\le3}=\left(1,0,1\right)\right)$}\tabularnewline
\cline{1-1} \cline{3-8} \cline{4-8} \cline{5-8} \cline{6-8} \cline{7-8} \cline{8-8} \cline{10-10} 
{\tiny{}$\left(1,1,0\right)$} &  &  &  &  &  &  & {\tiny{}$0$} &  & {\tiny{}$\mathbb{P}\left(\left(Z_{i}\right)_{i\le3}=\left(1,1,0\right)\right)$}\tabularnewline
\cline{1-1} \cline{3-8} \cline{4-8} \cline{5-8} \cline{6-8} \cline{7-8} \cline{8-8} \cline{10-10} 
{\tiny{}$\left(2,0,0\right)$} &  &  &  &  &  &  &  &  & {\tiny{}$\mathbb{P}\left(\left(Z_{i}\right)_{i\le3}=\left(2,0,0\right)\right)$}\tabularnewline
\cline{1-1} \cline{3-8} \cline{4-8} \cline{5-8} \cline{6-8} \cline{7-8} \cline{8-8} \cline{10-10} 
{\tiny{}$\vdots$} &  &  &  &  &  &  &  &  & {\tiny{}$\vdots$}\tabularnewline
\cline{1-1} \cline{3-8} \cline{4-8} \cline{5-8} \cline{6-8} \cline{7-8} \cline{8-8} \cline{10-10} 
\multicolumn{1}{c}{} & \multicolumn{1}{l}{} & \multicolumn{1}{c}{} & \multicolumn{1}{c}{} & \multicolumn{1}{c}{} & \multicolumn{1}{c}{} & \multicolumn{1}{c}{} & \multicolumn{1}{c}{} & \multicolumn{1}{c}{} & \multicolumn{1}{c}{}\tabularnewline
\cline{1-1} \cline{3-8} \cline{4-8} \cline{5-8} \cline{6-8} \cline{7-8} \cline{8-8} 
{\tiny{}Column sum} & {\tiny{}$\phantom{\bigg(}$} & {\tiny{}$1/6$} & {\tiny{}$1/6$} & {\tiny{}$1/6$} & {\tiny{}$1/6$} & {\tiny{}$1/6$} & {\tiny{}$1/6$} & \multicolumn{1}{c}{} & \multicolumn{1}{c}{}\tabularnewline
\cline{1-1} \cline{3-8} \cline{4-8} \cline{5-8} \cline{6-8} \cline{7-8} \cline{8-8} 
{\tiny{}$\mathcal{PM}\left(N\right)$} &  & {\tiny{}$0$} & {\tiny{}$0$} & {\tiny{}$\mathbb{P}\left(Z_{1}=Z_{2}=0\right)$} & {\tiny{}$\mathbb{P}\left(Z_{1}=Z_{2}=0\right)$} & {\tiny{}$\mathbb{P}\left(Z_{1}=Z_{2}=0\right)$} & {\tiny{}$\mathbb{P}\left(Z_{1}\le1\right)$} & \multicolumn{1}{c}{} & \multicolumn{1}{c}{}\tabularnewline
\cline{1-1} \cline{3-8} \cline{4-8} \cline{5-8} \cline{6-8} \cline{7-8} \cline{8-8} 
\end{tabular}{\tiny\par}
\end{figure}
\end{singlespace}

\noindent Each column with a pivot contains infinitely many pivots.
E.g., in Figure 2, column $\left(3,0,0\right)$ has a pivot in any
row of the form $\left(a,b,c\right)$ with $a\in\left\{ 0,1\right\} ,b,c\ge0$.
Columns labeled $\left(1,1,0\right)$ have a pivot in any row of the
form $\left(0,0,l\right)$ for any $l\in\mathbb{Z}_{\ge0}$. Moreover,
note that the independence of the process $\left(Z_{i}\left(3,x\right)\right)_{i\le3}$
allows us to distribute $\mathbb{P}$ through the parentheses in the
row sums and pivot mass expressions. The actual value of the row sum
and pivot masses depends on the choice of the process $\left(Z_{i}\left(3,x\right)\right)_{i\le3}$.
In $\S$4, we mention several choices for such processes $\left(Z_{i}\left(3,x\right)\right)_{i\le3}$
which will satisfy equation (\ref{eq:1}).
\end{example}
The following theorem plays a key role in the proof of Theorem 1.\footnote{When Theorem 2 is applied in $\S$4, additional indicator functions
will be included to remind us that $\mathbb{P}\left(Z_{i}\left(n,x\right)\le k\right)=0$
if $k<0$.} For convenience, in the proof of the following theorem, we simplify
the notation by writing $C_{i}\left(n\right)=C_{i}$, $Z_{i}\left(n,x\right)=Z_{i}$,
$M\left(n,x\right)=M$ and $N\left(n\right)=N$. Moreover, the notion
of pivot mass introduced in this section may be generalized; in a
particular setting, one should define pivot mass based on the constraints
required of their desired coupling. It is both a combinatorial and
probabilistic object since it is a sum of probability masses indexed
by the counting constraint (\ref{eq:2}).
\begin{thm}
(Pivot Mass Formula for 1 Column) Consider a fixed column label $N\left(n\right)\in A_{n}$
and denote its component counting process by $\left(C_{i}\left(n\right)\right)_{i\le n}$.
Its pivot mass is 
\begin{eqnarray*}
\mathcal{PM}\left(N\left(n\right)\right) & = & 1-\sum_{j\le n}\left(1_{\left\{ C_{j}>0\right\} }\left(1-\mathbb{P}\left(Z_{j}\le C_{j}-2\right)\right)\prod\limits _{\substack{i\not=j,\\
i\le n
}
}\left(1-\mathbb{P}\left(Z_{i}\le C_{i}-1\right)\right)\right)\\
 &  & +\left(\sum_{i\le n}1_{\left\{ C_{i}>0\right\} }-1\right)\prod_{i\le n}\left(1-\mathbb{P}\left(Z_{i}\le C_{i}-1\right)\right).
\end{eqnarray*}
\end{thm}
\begin{proof}
Given $1\le j\le n$, let $\overrightarrow{e_{j}}$ denote the row
vector of length $n$ whose $j$th entry is $1$ and whose other entries
are $0$. Given two vectors $\left(a_{i}\right)_{i\le n},\left(b_{i}\right)_{i\le n}$
in $\mathbb{R}^{n},$ we write $\left(a_{i}\right)_{i\le n}\le\left(b_{i}\right)_{i\le n}$
if $a_{i}\le b_{i}$ for each $i\le n$. Since $\sum_{k=1}^{\infty}\mathbb{P}\left(M\left(n,x\right)=k\right)=1$,
we have 
\begin{equation}
\mathcal{PM}\left(N\right)=1-\sum_{k:\left(k,N\right)\not\in P}\mathbb{P}\left(M=k\right).\label{eq:3}
\end{equation}
We have the event equality 
\begin{eqnarray*}
\left\{ \left(M,N\right)\not\in P\right\}  & = & \left\{ \exists j\le n:\left(Z_{i}\right)_{i\le n}\ge\left(C_{i}\right)_{i\le n}-\overrightarrow{e_{j}}\cdot1_{\left\{ C_{j}>0\right\} }\right\} 
\end{eqnarray*}
since the pair $\left(M,N\right)$ is a not pivot if and only if $Z_{i}\ge C_{i}$
for all $i$ except possibly one value $j$ with $Z_{j}=C_{j}-1$.
Since each $Z_{i},1\le i\le n$ is nonnegative, we can only have $Z_{j}=C_{j}-1$
when $C_{j}>0$. Note that if $Z_{i}\ge C_{i}$ for all $i$, then
any $j$ satisfies $\left(Z_{i}\right)_{i\le n}\ge\left(C_{i}\right)_{i\le n}-\overrightarrow{e_{j}}\cdot1_{\left\{ C_{j}>0\right\} }$.
On the other hand, if there exists a value $j$ for which $Z_{j}=C_{j}-1$
and $Z_{i}\ge C_{i}$ for all $i\not=j$, then $\left(Z_{i}\right)_{i\le n}\ge\left(C_{i}\right)_{i\le n}-\overrightarrow{e_{j}}\cdot1_{\left\{ C_{j}>0\right\} }$.
Therefore, the right hand side of equation (\ref{eq:3}) is
\begin{equation}
1-\sum_{k:\left(k,N\right)\not\in P}\mathbb{P}\left(M=k\right)=1-\mathbb{P}\left(\exists j\le n:\left(Z_{i}\right)_{i\le n}\ge\left(C_{i}\right)_{i\le n}-\overrightarrow{e_{j}}\cdot1_{\left\{ C_{j}>0\right\} }\right).\label{eq:4}
\end{equation}
We rewrite the probability $\mathbb{P}\left(\exists j\le n:\left(Z_{i}\right)_{i\le n}\ge\left(C_{i}\right)_{i\le n}-\overrightarrow{e_{j}}\cdot1_{\left\{ C_{j}>0\right\} }\right)$
by applying an inclusion-exclusion argument. Corresponding to any
$j\le n$ with $C_{j}>0,Z_{j}\ge C_{j}-1,$ and $Z_{i}\ge C_{i}$
for $i\not=j$, we add the term $\mathbb{P}\left(Z_{j}\ge C_{j}-1,\text{ and }Z_{i}\ge C_{i}\text{ for all }i\not=j\right)$.
As a result, we have added those elements with $Z_{i}\ge C_{i}$ for
all $i$ a total of $\sum_{i=1}^{n}1_{\left\{ C_{i}>0\right\} }$
many times. Therefore, we compensate by subtracting the term $\left(\sum_{i=1}^{n}1_{\left\{ C_{i}>0\right\} }-1\right)\mathbb{P}\left(\left(Z_{i}\right)_{i\le n}\ge\left(C_{i}\right)_{i\le n}\right)$.
Further, applying independence of the process $\left(Z_{i}\right)_{i\le n}$,
we have 
\begin{eqnarray*}
\mathbb{P}\left(Z_{j}\ge C_{j}-1\text{ and }Z_{i}\ge C_{i}\text{ for all }i\not=j\right) & = & \mathbb{P}\left(Z_{j}\ge C_{j}-1\right)\mathbb{P}\left(Z_{i}\ge C_{i}\text{ for all }i\not=j\right)\\
 & = & \mathbb{P}\left(Z_{j}\ge C_{j}-1\right)\prod\limits _{\substack{i\not=j,\\
i\le n
}
}\mathbb{P}\left(Z_{i}\ge C_{i}\right)
\end{eqnarray*}
and 
\begin{eqnarray*}
\mathbb{P}\left(\left(Z_{i}\right)_{i\le n}\ge\left(C_{i}\right)_{i\le n}\right) & = & \prod_{i\le n}\mathbb{P}\left(Z_{i}\ge C_{i}\right).
\end{eqnarray*}
Thus, the right hand side of equation (\ref{eq:4}) becomes 
\begin{equation}
1-\sum_{j\le n}\left(1_{\left\{ C_{j}>0\right\} }\mathbb{P}\left(Z_{j}\ge C_{j}-1\right)\prod\limits _{\substack{i\not=j,\\
i\le n
}
}\mathbb{P}\left(Z_{i}\ge C_{i}\right)\right)+\left(\sum_{i\le n}1_{\left\{ C_{i}>0\right\} }-1\right)\prod_{i\le n}\mathbb{P}\left(Z_{i}\ge C_{i}\right).\label{eq:5}
\end{equation}
Using the fact that $\mathbb{P}\left(Z_{i}\ge a\right)=1-\mathbb{P}\left(\left(Z_{i}\le a-1\right)\right)$,
expression (\ref{eq:5}) becomes 

\begin{eqnarray*}
 &  & 1-\sum_{j\le n}\left(1_{\left\{ C_{j}>0\right\} }\left(1-\mathbb{P}\left(Z_{j}\le C_{j}-2\right)\right)\prod\limits _{\substack{i\not=j,\\
i\le n
}
}\left(1-\mathbb{P}\left(Z_{i}\le C_{i}-1\right)\right)\right)\\
 &  & +\left(\sum_{i\le n}1_{\left\{ C_{i}>0\right\} }-1\right)\prod_{i\le n}\left(1-\mathbb{P}\left(Z_{i}\le C_{i}-1\right)\right).
\end{eqnarray*}
\end{proof}
The following result shows that only columns with label $\left(C_{i}\left(n\right)\right)_{i\le n}=\overrightarrow{e_{n}}$
have zero pivot mass. In this paper, we will only apply the $\left(\Leftarrow\right)$
part of the statement.\footnote{Note that $\left(\Rightarrow\right)$ implies that each column label
other than $\overrightarrow{e_{n}}$ has pivots. Using equations $\left(2.2\right)-\left(2.4\right)$
in $\S$2.2 of $\left[2\right]$ (which give the number of columns
with label $\overrightarrow{e_{n}}$ in each of these combinatorial
settings) we can always determine the number of columns that contain
pivots.}
\begin{thm}
For any nonempty collection $L\left(n\right)$ of column labels, $\mathcal{PM}\left(L\left(n\right)\right)=0$
if and only if a column with label $\left(C_{i}\left(n\right)\right)_{i\le n}=\overrightarrow{e_{n}}$
belongs to $L\left(n\right)$. 
\end{thm}
\begin{proof}
$\left(\Leftarrow\right)$ Given any row label $\left(Z_{i}\left(n,x\right)\right)_{i\le n}$,
the vector $\left(C_{i}\left(n\right)\right)_{i\le n}=\overrightarrow{e_{n}}$
satisfies 
\begin{eqnarray*}
\sum_{i\le n}\left(C_{i}\left(n\right)-Z_{i}\left(n,x\right)\right)^{+} & = & \left(C_{n}\left(n\right)-Z_{n}\left(n,x\right)\right)^{+}\\
 & = & \left(1-Z_{n}\left(n,x\right)\right)^{+}\\
 & \le & 1.
\end{eqnarray*}

\noindent Thus, $\mathcal{PM}\left(\overrightarrow{e_{n}}\right)=0$.
Therefore, given $\overrightarrow{e_{n}}\in L\left(n\right)$, we
have 
\begin{eqnarray*}
\mathcal{PM}\left(L\left(n\right)\right) & \le & \mathcal{PM}\left(\overrightarrow{e_{n}}\right)\\
 & = & 0.
\end{eqnarray*}
$\left(\Rightarrow\right)$ Now suppose $\overrightarrow{e_{n}}\not\in L\left(n\right)$.
Recall that any column label $\left(C_{i}\left(n\right)\right)_{i\le n}$
satisfies $\sum_{i\le n}iC_{i}\left(n\right)=n.$ Since $\overrightarrow{e_{n}}$
is the only column label with $\sum_{i\le n}C_{i}\left(n\right)=1$,
this gives us one of two cases for each column label in $L\left(n\right)$.
Either (a) there exists some $j$ with $C_{j}\left(n\right)\ge2$
or (b) there exists distinct $j,k$ with $C_{j}\left(n\right)\ge1,C_{k}\left(n\right)\ge1$.
In case (a), using any row label $\left(Z_{i}\left(n,x\right)\right)_{i\le n}$
with $Z_{j}\left(n,x\right)=0$, we have 
\begin{eqnarray*}
\sum_{i\le n}\left(C_{i}\left(n\right)-Z_{i}\left(n,x\right)\right)^{+} & \ge & C_{j}\left(n\right)-Z_{j}\left(n,x\right)\\
 & \ge & 2.
\end{eqnarray*}
In case (b), we can take any $\left(Z_{i}\left(n,x\right)\right)_{i\le n}$
with $Z_{j}\left(n,x\right)=Z_{k}\left(n,x\right)=0$ to ensure that
\begin{eqnarray*}
\sum_{i\le n}\left(C_{i}\left(n\right)-Z_{i}\left(n,x\right)\right)^{+} & \ge & \left(C_{j}\left(n\right)-Z_{j}\left(n,x\right)\right)+\left(C_{k}\left(n\right)-Z_{k}\left(n,x\right)\right)\\
 & \ge & 2.
\end{eqnarray*}
Since we have just showed that each column label other than $\overrightarrow{e_{n}}$
has a pivot, we use the fact that each of these columns has a pivot
in the first row (labeled $\left(Z_{i}\left(n,x\right)\right)_{i\le n}=\left(0,0,\ldots,0\right)$).
Note that $\mathbb{P}\left(M\left(n,x\right)=i\right)>0$ for all
distributions in this paper (see $\S$4), so we have 
\begin{eqnarray*}
\mathcal{PM}\left(L\left(n\right)\right) & \ge & \mathbb{P}\left(Z_{i}\left(n,x\right)=0,\text{ }\forall i\le n\right)\\
 & = & \mathbb{P}\left(M\left(n,x\right)=1\right)\\
 & > & 0.
\end{eqnarray*}
\end{proof}

\section{Pivot Mass can be made Arbitrarily Small for Assemblies, Multisets,
and Selections}

The following condition on $\mathcal{PM}$ will be verified for our
three combinatorial structures:
\begin{equation}
\forall n\in\mathbb{N}\;\,\forall\varepsilon>0\;\,\exists x\left(n\right):x>x\left(n\right)\implies\ensuremath{\,\,\text{equation \eqref{eq:1}}\text{ holds and }\mathcal{PM}_{\left(n,x\right)}\left(\cdot\right)<\varepsilon.}\label{eq:6}
\end{equation}

\subsection{Assemblies}

In the assembly setting, we can take\footnote{Although the distribution of $Z_{i}\left(n,x\right)$ does not depend
on $n$, the choice the process $\left(Z_{i}\left(n,x\right)\right)_{i\le n}$
satisfying (\ref{eq:1}) does depend on $n$. I.e., if $\left(Z_{i}\left(n,x\right)\right)_{i\le n}$
and $\left(Z_{i}\left(n+1,x\right)\right)_{i\le n+1}$ equal $\left(C_{i}\left(n\right)\right)_{i\le n}$
and $\left(C_{i}\left(n+1\right)\right)_{i\le n+1},$ respectively,
conditional on the events $\left\{ \sum_{i\le n}iZ_{i}=n\right\} $
and $\left\{ \sum_{i\le n+1}iZ_{i}=n+1\right\} $, respectively, then
we need not have $\left(Z_{i}\left(n,x\right)\right)_{i\le n}=\left(Z_{i}\left(n+1,x\right)\right)_{i\le n}$.} $Z_{i}\left(n,x\right)\sim\text{Po}\left(\frac{m_{i}x^{i}}{i!}\right)$
for any $x>0$ to obtain equation (\ref{eq:1}) ($\S$2.3 of \cite{key-2}).
Recall that the CDF of a random variable $Z\sim\text{Po}\left(\lambda\right)$
is given by $\mathbb{P}\left(Z\le k\right)=\frac{\Gamma\left(\left\lfloor k+1\right\rfloor ,\lambda\right)}{\left\lfloor k\right\rfloor !}$
for $k\in\mathbb{Z}_{\ge0}$, where $\Gamma\left(a,b\right)$ is the
\textbf{upper incomplete gamma function} \textendash{} i.e., $\Gamma\left(a,b\right)=\int_{b}^{\infty}t^{a-1}e^{-t}dt$. 
\begin{lem}
For a fixed $a>0,$ we have $\lim_{b\to\infty}\Gamma\left(a,b\right)=0.$
\end{lem}
\begin{proof}
Since $\Gamma\left(a,0\right)=\Gamma\left(a\right)$ is convergent
for $a>0$, we have 
\begin{eqnarray*}
\Gamma\left(a,b\right) & = & \Gamma\left(a\right)-\int_{0}^{b}t^{a-1}e^{-t}dt\\
 & \to & \Gamma\left(a\right)-\Gamma\left(a\right)\text{ as }b\to\infty\\
 & = & 0.
\end{eqnarray*}
\end{proof}
\noindent We can use Lemma 1 and take $x\to\infty$ to obtain 
\begin{equation}
\frac{\Gamma\left(C_{i},\frac{m_{i}x^{i}}{i!}\right)}{\left(C_{i}-1\right)!}\to0\text{ when }C_{i}>0\label{eq:7}
\end{equation}
and 
\begin{equation}
\frac{\Gamma\left(C_{j}-1,\frac{m_{j}x^{j}}{j!}\right)}{\left(C_{j}-2\right)!}\to0\text{ when }C_{j}>1.\label{eq:8}
\end{equation}
Therefore, Theorem 2 implies that $\mathcal{PM}\left(N\left(n\right)\right)$
equals 
\begin{eqnarray*}
 &  & 1-\sum_{j\le n}\left(1_{\left\{ C_{j}>0\right\} }\left(1-1_{\left\{ C_{j}>1\right\} }\frac{\Gamma\left(C_{j}-1,\frac{m_{j}x^{j}}{j!}\right)}{\left(C_{j}-2\right)!}\right)\prod\limits _{\substack{i\not=j,\\
i\le n
}
}\left(1-1_{\left\{ C_{i}>0\right\} }\frac{\Gamma\left(C_{i},\frac{m_{i}x^{i}}{i!}\right)}{\left(C_{i}-1\right)!}\right)\right)\\
 &  & +\left(\sum_{i\le n}1_{\left\{ C_{i}>0\right\} }-1\right)\prod_{i\le n}\left(1-1_{\left\{ C_{i}>0\right\} }\frac{\Gamma\left(C_{i},\frac{m_{i}x^{i}}{i!}\right)}{\left(C_{i}-1\right)!}\right).
\end{eqnarray*}
If we let $x\to\infty$, we can apply (\ref{eq:7}) and (\ref{eq:8})
to deduce that 
\begin{eqnarray*}
\mathcal{PM}\left(N\left(n\right)\right) & \to & 1-\sum_{j\le n}\left(1_{\left\{ C_{j}>0\right\} }\left(1-0\right)\prod\limits _{\substack{i\not=j,\\
i\le n
}
}\left(1-0\right)\right)\\
 &  & +\left(\sum_{i\le n}1_{\left\{ C_{i}>0\right\} }-1\right)\prod_{i\le n}\left(1-0\right)\\
 & = & 1-\sum_{j\le n}1_{\left\{ C_{j}>0\right\} }+\left(\sum_{i\le n}1_{\left\{ C_{i}>0\right\} }-1\right)\\
 & = & 0.
\end{eqnarray*}
This verifies condition (\ref{eq:6}) for assemblies.

\subsection{Multisets}

In the multiset setting, we can take $Z_{i}\left(n,x\right)\sim\text{\text{NB}}\left(m_{i},x^{i}\right),$
for any $x\in\left(0,1\right)$, to obtain equation (\ref{eq:1})
($\S$2.3 of \cite{key-2}). Recall that the CDF of $Z\sim\text{NB}\left(r,p\right)$
is given by $\mathbb{P}\left(Z\le k\right)=1-I_{p}\left(k+1,r\right),$
where $I_{p}$ is the \textbf{regularized incomplete beta function}.
That is, $I_{x}\left(a,b\right)=\frac{B\left(x;a,b\right)}{B\left(a,b\right)}$,
where $B\left(a,b\right)=\int_{0}^{1}t^{a-1}\left(1-t\right)^{b-1}dt$,
defined for $\text{Re}\left(a\right)>0$ and $\text{Re}\left(b\right)>0$,
is the \textbf{beta function} and $B\left(x;a,b\right)=\int_{0}^{x}t^{a-1}\left(1-t\right)^{b-1}dt$
is the \textbf{incomplete beta function}. 
\begin{lem}
Given $a>0$, $\lim_{x\to1}I_{x}\left(a,b\right)=1.$
\end{lem}
\begin{proof}
We have 
\begin{eqnarray*}
\lim_{x\to1}I_{x}\left(a,b\right) & = & \lim_{x\to1}\frac{B\left(x;a,b\right)}{B\left(a,b\right)}\\
 & = & \lim_{x\to1}\frac{\int_{0}^{x}t^{a-1}\left(1-t\right)^{b-1}dt}{\int_{0}^{1}t^{a-1}\left(1-t\right)^{b-1}dt}\\
 & = & \frac{\int_{0}^{1}t^{a-1}\left(1-t\right)^{b-1}dt}{\int_{0}^{1}t^{a-1}\left(1-t\right)^{b-1}dt}\\
 & = & 1.
\end{eqnarray*}
\end{proof}
\noindent Using Theorem 2, $\mathcal{PM}\left(N\left(n\right)\right)$
equals 
\begin{eqnarray*}
 &  & 1-\Bigg(\sum_{j\le n}\left(1_{\left\{ C_{j}>0\right\} }\left(1-1_{\left\{ C_{j}>1\right\} }\left(1-I_{x^{j}}\left(C_{j}-1,m_{j}\right)\right)\right)\prod\limits _{\substack{i\not=j,\\
i\le n
}
}\left(1-1_{\left\{ C_{i}>0\right\} }\left(1-I_{x^{i}}\left(C_{i},m_{i}\right)\right)\right)\right)\\
 &  & +\left(\sum_{i\le n}1_{\left\{ C_{i}>0\right\} }-1\right)\prod_{i\le n}\left(1-1_{\left\{ C_{i}>0\right\} }\left(1-I_{x^{i}}\left(C_{i},m_{i}\right)\right)\right).
\end{eqnarray*}
Taking $x\to1$ and applying Lemma 2, we have 
\begin{eqnarray*}
\mathcal{PM}\left(N\left(n\right)\right) & \to & 1-\Bigg(\sum_{j\le n}\left(1_{\left\{ C_{j}>0\right\} }\left(1-1_{\left\{ C_{j}>1\right\} }\left(1-1\right)\right)\prod\limits _{\substack{i\not=j,\\
i\le n
}
}\left(1-1_{\left\{ C_{i}>0\right\} }\left(1-1\right)\right)\right)\\
 &  & +\left(\sum_{i\le n}1_{\left\{ C_{i}>0\right\} }-1\right)\prod_{i\le n}\left(1-1_{\left\{ C_{i}>0\right\} }\left(1-1\right)\right)\\
 & = & 1-\sum_{j\le n}1_{\left\{ C_{j}>0\right\} }+\left(\sum_{i\le n}1_{\left\{ C_{i}>0\right\} }-1\right)\\
 & = & 0,
\end{eqnarray*}
which verifies condition (\ref{eq:6}) for multisets.

\subsection{Selections}

In the selection setting, we can take $Z_{i}\left(n,x\right)\sim\text{Bin}\left(m_{i},\frac{x^{i}}{1+x^{i}}\right)$
with $0<x<\infty$ in order to obtain equation (\ref{eq:1}) ($\S$2.3
of \cite{key-2}). In our case, we are taking $p=\frac{x^{i}}{1+x^{i}},$
so $p\to1$ if and only if $x\to\infty$. Recall that the CDF of $Z\sim\text{Bin}\left(n,p\right)$
is given by $\mathbb{P}\left(Z\le k\right)=I_{1-p}\left(n-k,1+k\right)$.
Using Theorem 2, we can express $\mathcal{PM}\left(N\left(n\right)\right)$
as 
\begin{eqnarray*}
 &  & 1-\sum_{j\le n}\left(1_{\left\{ C_{j}>0\right\} }\left(1-1_{\left\{ C_{j>1}\right\} }I_{1-p}\left(m_{j}-C_{j}+2,C_{j}-1\right)\right)\prod\limits _{\substack{i\not=j,\\
i\le n
}
}\left(1-1_{\left\{ C_{i}>0\right\} }I_{1-p}\left(m_{i}-C_{i}+1,C_{i}\right)\right)\right)\\
 &  & +\left(\sum_{i\le n}1_{\left\{ C_{i}>0\right\} }-1\right)\prod_{i\le n}\left(1-1_{\left\{ C_{i}>0\right\} }I_{1-p}\left(m_{i}-C_{i}+1,C_{i}\right)\right).
\end{eqnarray*}

\begin{lem}
We have $\lim_{p\to1}I_{1-p}\left(n-k,1+k\right)=0$.
\end{lem}
\begin{proof}
\begin{eqnarray*}
\lim_{p\to1}I_{1-p}\left(n-k,1+k\right) & = & \lim_{p\to1}\frac{B\left(1-p;n-k,1+k\right)}{B\left(n-k,1+k\right)}\\
 & = & \lim_{p\to1}\frac{\int_{0}^{1-p}t^{n-k-1}\left(1-t\right)^{k}}{\int_{0}^{1}t^{n-k-1}\left(1-t\right)^{k}}\\
 & = & 0.
\end{eqnarray*}
\end{proof}
\noindent Using Lemma 3, we see that 
\begin{equation}
I_{1-p}=I_{1-\frac{x^{i}}{1+x^{i}}}\to0\label{eq:9}
\end{equation}

\noindent if $x\to\infty$. Thus, we apply Theorem 2 and Lemma 3 while
taking $x\to\infty$ to obtain 
\begin{eqnarray*}
\mathcal{PM}\left(N\left(n\right)\right) & \stackrel{\eqref{eq:9}}{\to} & 1-\sum_{j\le n}1_{\left\{ C_{j}>0\right\} }+\left(\sum_{i\le n}1_{\left\{ C_{i}>0\right\} }-1\right)\\
 & = & 0,
\end{eqnarray*}
which verifies condition (\ref{eq:6}) for selections.

\section{Using Pivot Mass to Provide Couplings}

Let $S$ and $T$ be complete separable metric spaces. Denote by $p_{S}$
the projection of $S\times T$ onto $S$. Let $\omega$ be a nonempty
closed subset of $S\times T$ and $\varepsilon\ge0$. The following
result is Theorem 11 of \cite{key-4}. 
\begin{thm}
\footnote{We thank Anthony Quas for suggesting the use of Hall's Marriage Theorem.
Strassen's Theorem is a variant of the marriage theorem.}(Strassen) There is a probability measure $\lambda$ in $S\times T$
with marginals $\mu$ and $\nu$ such that $\lambda\left(\omega\right)\ge1-\varepsilon$,
if and only if for all closed sets $L\subseteq T$
\begin{equation}
\nu\left(L\right)\le\mu\left(p_{s}\left(\omega\cap\left(S\times L\right)\right)\right)+\varepsilon.\label{eq:10}
\end{equation}
\end{thm}

\begin{proof}[Proof of Theorem 1]

Let us define 
\begin{align*}
S & =\left(\mathbb{Z}_{\ge0}\right)^{n},\\
T & =\left\{ \left(a_{i}\right)_{i\le n}\in\left(\mathbb{Z}_{\ge0}\right)^{n}:\sum_{i\le n}ia_{i}=n\right\} ,
\end{align*}
corresponding to the set of row labels and the set of column labels
respectively, and endow both $S$ and $T$ with the metric $d$ on
$\mathbb{Z}^{n}$ defined as 
\[
d\left(\left(x_{i}\right)_{i\le n},\left(y_{i}\right)_{i\le n}\right)\coloneqq\max_{i\le n}\left|x_{i}-y_{i}\right|.
\]
Since $S$ is finite and $T$ is countably infinite, both $S$ and
$T$ are separable. In both $S$ and $T$ we have 
\begin{equation}
\left(x_{i}\right)_{i\le n}\not=\left(y_{i}\right)_{i\le n}\implies d\left(\left(x_{i}\right)_{i\le n},\left(y_{i}\right)_{i\le n}\right)\ge1\label{eq:11}
\end{equation}
since our $n$-tuples are integer-valued. Therefore, every Cauchy
sequence in $S$ (or in $T$) converges in $S$ (or in $T$). Thus,
$S$ and $T$ are complete. Our goal is to apply Theorem 4 with $\varepsilon=0$
and 
\begin{eqnarray*}
\omega & = & P^{c},\\
L & = & L\left(n\right),\\
\mu_{\left(n,x\right)}\left(i\right) & = & \mathbb{P}\left(M\left(n,x\right)=i\right),i\in S,\\
\nu_{n}\left(j\right) & = & \mathbb{P}\left(N\left(n\right)=j\right),j\in T,\\
\lambda & = & p,
\end{eqnarray*}
where $P=\left\{ \left(i,j\right)\in S\times T:\left(i,j\right)\text{ is a pivot}\right\} $,
$L\left(n\right)$ denotes an arbitrary subset of $T$, and $p$ is
our desired joint PMF, with marginals corresponding to $M\left(n,x\right)$
and $N\left(n\right)$, such that $\left(i,j\right)\in P$ implies
$p\left(i,j\right)=0$. Let us endow $\omega$ with the metric $d_{\omega}$
obtained by restricting the metric
\begin{equation}
d_{S\times T}\left(\left(\left(s_{i}\right)_{i\le n},\left(t_{i}\right)_{i\le n}\right),\left(\left(s_{i}'\right)_{i\le n},\left(t_{i}'\right)_{i\le n}\right)\right)\coloneqq\max\left(d\left(\left(s_{i}\right)_{i\le n},\left(s_{i}'\right)_{i\le n}\right),d\left(\left(t_{i}\right)_{i\le n},\left(t_{i}'\right)_{i\le n}\right)\right)\label{eq:12}
\end{equation}
on $S\times T$ to $\omega$. To show that $\omega$ is closed, we
first show that $S$ and $T$ are closed. The set $T$ is closed since
it is finite. Suppose that 
\[
\left(\left(s_{i}\left(k\right)\right)_{i\le n}\right)_{k\in\mathbb{N}}
\]
is a sequence of $n$-tuples $\left(s_{i}\left(k\right)\right)_{i\le n}\in S$
with 
\[
\lim_{k\to\infty}\left(s_{i}\left(k\right)\right)_{i\le n}=l_{1}
\]
for some $n$-tuple $l_{1}\in\mathbb{Z}^{n}$. To show that $S$ is
closed, it suffices to show that $l_{1}\in S$. For all $\varepsilon'\in\left(0,1\right)$
there exists a constant $K\in\mathbb{N}$ such that 
\[
k>K\implies d\left(\left(s_{i}\left(k\right)\right)_{i\le n},l_{1}\right)<\varepsilon'.
\]
Since $\varepsilon'<1$, (\ref{eq:11}) implies
\[
l_{1}=\left(s_{i}\left(K+1\right)\right)_{i\le n},
\]
so $l_{1}\in S$. Therefore, $S$ is closed. Now to show that $\omega$
is closed in $S\times T$, suppose that 
\[
\left(\left(s_{i}\left(k\right)_{i\le n}\right),\left(t_{i}\left(k\right)\right)_{i\le n}\right)_{k\in\mathbb{N}}
\]
is a sequence of pairs 
\[
\left(\left(s_{i}\left(k\right)_{i\le n}\right),\left(t_{i}\left(k\right)\right)_{i\le n}\right)\in\omega
\]
of $n$-tuples $s_{i}\left(k\right)_{i\le n}\in S,\left(t_{i}\left(k\right)\right)_{i\le n}\in T$
with 
\[
\lim_{k\to\infty}\left(\left(s_{i}\left(k\right)_{i\le n}\right),\left(t_{i}\left(k\right)\right)_{i\le n}\right)=\left(l_{1},l_{2}\right).
\]
for some $n$-tuples $l_{1},l_{2}\in\mathbb{Z}^{n}$. Since $S$ and
$T$ are closed, we have $l_{1}\in S$ and $l_{2}\in T$. For all
$\varepsilon'\in\left(0,1\right)$ there exists a constant $K\in\mathbb{N}$
such that 
\[
k>K\implies d_{S\times T}\left(\left(s_{i}\left(k\right)_{i\le n}\right),\left(l_{1},l_{2}\right)\right)<\varepsilon'.
\]
Therefore, 
\[
k>K\overset{\left(12\right)}{\implies}d\left(\left(s_{i}\left(k\right)\right)_{i\le n},l_{1}\right),d\left(\left(t_{i}\left(k\right)\right)_{i\le n},l_{2}\right)<\varepsilon'.
\]
Since $\varepsilon'<1$, we have
\[
k>K\overset{\left(11\right)}{\implies}d\left(\left(s_{i}\left(k\right)\right)_{i\le n},l_{1}\right)=d\left(\left(t_{i}\left(k\right)\right)_{i\le n},l_{2}\right)=0.
\]
Therefore, applying (\ref{eq:11}) twice, we obtain
\[
\left(l_{1},l_{2}\right)=\left(\left(s_{i}\left(K+1\right)\right)_{i\le n},\left(t_{i}\left(K+1\right)\right)_{i\le n}\right)\in\omega,
\]
so $\omega$ is a closed subset of $S\times T$.

Further, $\omega\not=\emptyset$ since given any column label $j$,
the pair $\left(j,j\right)$ belongs to $\omega$. Note that the set
$L\left(n\right)$ is a closed subset of the column labels since $L\left(n\right)$
is a finite. Moreover,

\begin{eqnarray*}
v_{n}\left(L\left(n\right)\right) & = & \mathbb{P}\left(N\left(n\right)\in L\left(n\right)\right)\\
 & = & \frac{\#L\left(n\right)}{k_{n}},
\end{eqnarray*}
and

\begin{eqnarray*}
\mu_{\left(n,x\right)}\left(p_{s}\left(\omega\cap\left(S\times U\right)\right)\right) & = & \mathbb{P}\left(M\in p_{S}\left(\omega\cap\left(S\times U\right)\right)\right)\\
 & = & \mathbb{P}\left(M\in p_{S}\left(P^{c}\cap\left(S\times L\left(n\right)\right)\right)\right)\\
 & = & \mathbb{P}\left(\exists j\in L\left(n\right):\left(M,j\right)\not\in P\right)\\
 & = & 1-\mathcal{PM}_{\left(n,x\right)}\left(L\left(n\right)\right).
\end{eqnarray*}
Therefore, (\ref{eq:10}) is equivalent to 
\[
\frac{\#L\left(n\right)}{k_{n}}\le1-\mathcal{PM}_{\left(n,x\right)}\left(L\left(n\right)\right).
\]
The latest inequality is equivalent to 
\begin{equation}
\mathcal{PM}_{\left(n,x\right)}\left(L\left(n\right)\right)\le1-\frac{\#L\left(n\right)}{k_{n}}.\label{eq:13}
\end{equation}

\noindent By (\ref{eq:6}), the left hand side can be made arbitrarily
small, so (\ref{eq:13}) holds when $1-\frac{\#L\left(n\right)}{k_{n}}>0$.
When $1-\frac{\#L\left(n\right)}{k_{n}}=0$, we must have $L\left(n\right)=A_{n},$
so that $\mathcal{PM}_{\left(n,x\right)}\left(L\left(n\right)\right)=0$
by Theorem 3. Therefore, by the conclusion of Strassen's Theorem,
there exists a joint probability measure $p$, with marginals $\mathbb{P}\left(M\left(n,x\right)=\cdot\right)$
and $\mathbb{P}\left(N\left(n\right)=\cdot\right)$, such that $p\left(\omega\right)=1$.
I.e, the probability of having no pivot in this joint distribution
is $1$. Hence, the proof of Theorem 1 is complete.

\end{proof}

\end{document}